\RequirePackage{fix-cm}
\documentclass[smallextended]{svjour3}       
\smartqed  

\usepackage{graphicx}
\usepackage{amssymb}
\usepackage[active]{srcltx}
\usepackage{a4wide}
\usepackage{amsmath}
\usepackage{amssymb}
\usepackage{amstext}
\usepackage[normalem]{ulem}
\usepackage{cite}
\usepackage{epsfig}
\usepackage{enumerate}
\usepackage{color}
\usepackage{amssymb}
\allowdisplaybreaks
\usepackage{float}
\usepackage[linkcolor=blue, urlcolor=blue, citecolor=blue,
colorlinks, bookmarks]{hyperref}
\usepackage{graphicx,palatino,pifont,times}
\usepackage{amssymb}
\usepackage{amstext}
\usepackage{cite}
\usepackage{enumerate}
\usepackage{bm}
\usepackage{scalerel,stackengine}
\stackMath
\newcommand\reallywidehat[1]{%
	\savestack{\tmpbox}{\stretchto{%
			\scaleto{%
				\scalerel*[\widthof{\ensuremath{#1}}]{\kern-.6pt\bigwedge\kern-.6pt}%
				{\rule[-\textheight/2]{1ex}{\textheight}}
			}{\textheight}%
		}{0.5ex}}%
	\stackon[1pt]{#1}{\tmpbox}%
}

\newcommand{\be}{\begin{equation}}
	\newcommand{\ee}{\end{equation}}

\begin{document}
	
	\title{Generalized multivariate Fractal Interpolation Function and $\alpha$-Fractal Function}
	
	\author{Megha Pandey
		\and Pavjeet Singh$^*$ 
		\and S. K. Katiyar
		\and T. Som
	}
	
	
	\institute{Megha Pandey \at
		School of Mathematics, Northwest University Xi'an, Shaanxi Province, 710069, China\\
		\email{meghapandey1071996@gmail.com}
		\and
		Pavjeet Singh\at
		DR B R Ambedkar National Institute of technology, Jalandhar, India, 144011\\
	*Correspondence to	\email{pavjeetsingh23@gmail.com}
		\and
		S. K. Katiyar\at
		DR B R Ambedkar National Institute of technology, Jalandhar, India, 144011\\
		\email{sbhkatiyar@gmail.com}
		\and
		T. Som\at
		Indian Institute of Technology (Banaras Hindu University), Varanasi, India, 221005\\
		\email{tsom.apm@itbhu.ac.in}             
	}

	\maketitle
	\begin{abstract}
		In this paper, we introduce a new approach for constructing multivariate fractal interpolation functions and $\alpha$-fractal functions associated with multivariate functions. Unlike the existing methods that rely on the Banach contraction principle, here the construction is based on Matkowski and Rakotch contractions. While numerous methods for constructing multivariate fractal interpolation functions have been explored in the literature, the approach given in this paper is distinct in the sense that it generalizes all previously known techniques and provides a broader framework for such constructions. We propose a technique to develop nonlinear iterated function systems using the generalized contractions and establish that the attractors of such systems are the graphs of continuous multivariate functions interpolating theoretical data points. Furthermore, for the Rakotch contractions, the existence of an invariant Borel probability measure supported on the graph of the associated multivariate fractal interpolation function is explored. \\

	\end{abstract}

	\section{Introduction}
	Interpolation: A process of estimating the value of a function at a point from its values at nearby points or a technique to construct new data points within the range of a discrete set of known data points. Numerous interpolation schemes have been developed in the literature. Most of the given methods produce interpolants that are differentiable at almost every point of the domain, and consequently, these interpolants are smooth functions. 
	But in real life, many functions do not observe smoothness in their nature; for instance, we encounter many experimental signals that are tangled and rarely smooth. Therefore, to interpolate these types of functions, we need nonsmooth interpolants. To deal with such issues, in 1986 Barnsley\cite{barnsley1986fractal} introduced the concept of the Fractal Interpolation Function (FIF). To construct the FIF, Barnsley used the idea of an Iterated Function System (IFS). An FIF is a continuous function such that its graph is the attractor of an appropriately chosen IFS. Following this seminal work of Barnsley, many researchers have later generalized the notion of FIF. Numerous types of FIF have been explored in the literature so far. For instance,
	in \cite{barnsley1989hidden}, Barnsley et al. studied the concept of hidden variable fractal interpolation functions. In \cite{hardin1993fractal}, Hardin et al. explored the FIFs from $\mathbb{R}^n$ to $\mathbb{R}^m$ and their projections. In \cite{navascues2005fractal}, Navascues et al. introduced a new type of fractal interpolation function associated with a continuous real-valued univariate function, known as the $\alpha$-fractal function. It provides a systematic way to generate self-referential perturbations of a given continuous function called the base function. These functions are obtained using an IFS, where the scale factor $\alpha$ controls the fractal characteristics of the perturbation. Extending this idea to the multivariate case, the base function is defined on a compact domain in $\mathbb{R}^n$. In addition, the domain is divided into sub regions, and $\alpha$ becomes a multi-indexed parameter regulating local scaling. 
	In \cite{massopust2014fractal}, Massopust has introduced the concept of fractal functions and fractal surfaces. In \cite{verma2023dimensions}, Verma and Priyadarshi developed FIFs using Rakotch contractions and proved the existence of an invariant Borel measure supported on their graphs, along with dimensional estimates.
	In \cite{prithvi2022interpolative,prithvi2023generalized}, Prithvi and Katiyar introduced interpolative operators that connect classical fractal interpolation with multivalued fractal structures, thereby broadening the analytical and constructive framework of fractal functions. Their investigations on generalized Kannan-type mappings further established new contractive conditions applicable to IFS, ensuring the existence of attractors. In \cite{prithvi2025comments}, they examined the approximation aspects of fractal functions associated with Reich contractions. Moreover, their work on nonconventional IFSs \cite{prithvi2023revisiting} revisited and refined the Hutchinson–Barnsley framework, leading to richer classes of fractal attractors and enhanced applicability of fractal interpolation theory. In a related direction \cite{ullah2023generalized}, Ullah and Katiyar developed the concept of generalized $G$-Hausdorff spaces, providing an extended metric setting suitable for studying fractal sets under weaker structural assumptions.
	
	The study of fractal functions has alreagarnered considerable attention due to its applications in various real-life and scientific areas. Apart from constructing various types of FIF, people are focusing on analyzing its properties. Some vital characteristiincluding calculus, dimension, stability, perturbation error, andmoothness, have been extensively explored. For example, in \cite{barnsley2015bilinear}, Barnsley has reported the box dimension of the bilinear FIF. In \cite{akhtar2017box}, the concept of the box dimension of the $\alpha$-fractal function has been explored. In \cite{pandey2022fractal}, Pandey et al. have constructed a multivariate $\alpha$-fractal function and established some results on its dimension, and in \cite{pandey2022set}, the idea of set-valued $\alpha$-fractal functions has been proposed.
	
	\subsection{Motivation and Work done}
	The set of IFS is a powerful tool for producing and inspecting the FIF. The Banach contraction principle was vital in the set of IFS, which is used to construct the FIF in all the above-referred research. In the sequel, researchers have explored using the well-known fixed point results in the fixed point theory literature to construct the FIF in a more general setting, i.e., they have tried to use the weaker condition on IFS, which is used to produce the FIF. In \cite{strobin2013certain}, Strobin and Swaczyna studied certain generalizations of IFS, but the existence of fractal interpolation corresponding to those IFS was not known at that time. In 1962, Rakotch \cite{rakotch1962note} introduced the first substantial generalization of the Banach contraction principle in this direction. After that, some methods to construct the nonlinear FIF was initiated using generalized fixed point results instead of the Banach fixed point theorem. For example, in \cite{ri2017new,ri2018new,ri2019new}, Rakotch or Geraghty fixed point theory is applied. In \cite{ri2018newidea}, the idea of Matkowski and Rakotch's fixed point theory is used. 
	The idea of the construction of FIF using Rakotch or Geraghty, or Matkowski \cite{matkowski1975integrable} fixed point theory has been given for univariate and bivariate functions.  
	
	\noindent
	In this article, we have proposed the construction of FIF using Rakotch and Matkowski's fixed point theory for multivariate functions. We have shown the existence of multivariate FIF corresponding to the set of IFS that satisfy the Matkowski and the Rakotch contraction. Furthermore, we have demonstrated the formulation of the associated $\alpha$-fractal function arising from these contractions. In addition, we have established the existence of an invariant probability measure for IFSs satisfying Rakotch contractions and claimed that its support coincides with the associated attractor.

	\subsection{Delineation}
	The proposed paper is organized as follows. In the next section, we present the notations and preliminaries used in our paper. Section \ref{sec2} is dedicated to the construction of multivariate FIFs using Matkowski and Rakotch contractions. We have proved that the FIF constructed by this method is an attractor of a suitable IFS, where the set of IFS satisfies the Matkowski and Rakotch contractions, respectively. In Section \ref{sec4}, we construct $\alpha$-fractal functions corresponding to multivariate continuous functions using Matkowski and Rakotch contractions. In Section~\ref{sec5}, we investigate invariant probability measures corresponding to IFSs generated by Rakotch contractions and establish the existence of such measures supported on the graph of the multivariate FIF. Finally, the paper is concluded in Section~\ref{sec3} with some future directions.

	\section{Preliminaries and Notations}
	\subsection{Notations}
	Throughout this paper, we shall use the following notations:
	\begin{itemize}
		\item $\mathbb{N}$ : set of natural numbers
		\item $\mathbb{R}$ : set of real numbers
		\item $\mathbb{R}_{+}$: set of non-negative real numbers
		\item $I$ : unit interval $[0,1]$
		\item $I^q=I\times \cdots \times I$: $q$-times Cartesian product of $[0,1]$, where $q\in \mathbb{N}$
		\item $\mathcal{C}(I^q)$ : set of all real-valued continuous functions on $I^q$
		\item  $\mathcal{K}(\mathbb{R})=\{A\subset \mathbb{R}:A~\text{is a compact subset of}~\mathbb{R}\}$
		\item $\mathcal{H}_d(A,B)=\max \left\{\underset{a\in A}{\sup}\underset{b\in B}{\inf}\lvert a-b\rvert,\underset{b\in B}{\sup}\underset{a\in A}{\inf}\lvert b-a\rvert\right\}$ be the metric defined on $\mathcal{K}(\mathbb{R})$. It is broadly known as the Hausdorff metric.
	\end{itemize}
	
	\subsection{Preliminaries}
	\begin{definition}
		Let $X$ and $Y$ be metric spaces and $h:X\rightarrow Y$ be a function from $X$ to $Y$, then the graph of $h$ is defined as,
		\begin{equation}\label{Gf1}
			\mathcal{G}(h)=\{(x,h(x))~:~ x\in X \}.
		\end{equation}
	\end{definition}
	
	\begin{definition}
		Let $(X,d)$ be a metric space and $f: X\rightarrow X$ be a self-map. If for some function $\phi: \mathbb{R}_{+}\rightarrow \mathbb{R}_{+}$, we have
		\[d(f(x),f(y))\leq \phi(d(x,y)) \text{ for all } x,y\in X,\]
		then $f$ is said to be $\phi$-contraction.
	\end{definition}

	\begin{definition}
		Let $(X,d)$ be a metric space and $f$ be a self-map defined on $X$. If $f$ is a $\phi$-contraction such that $\phi:\mathbb{R}_{+}\rightarrow \mathbb{R}_{+}$ is a non-decreasing function and for every $t>0$, $\phi^{n}(t)\to 0$ as $n\to \infty$ ($\phi^n$ is the $n^{th}$ iteration of $\phi$). Then, $f$ is called the Matkowski contraction.
	\end{definition}
	
	\begin{definition}
		If $f$ is a $\phi$-contraction such that for any $t>0$, $\beta(t)=\frac{\phi(t)}{t}<1$ and the function $\beta$ is a non-increasing function. Then, $f$ is a Rakotch contraction.
	\end{definition}
	\begin{remark}
		It is easy to find that every Rakotch contraction map is a Matkowski contraction map.
	\end{remark}
	\begin{remark}\cite[Theorem $1$]{jachymski2007nonlinear}\label{concave}
		The map $f:X \rightarrow X$ is a Rakotch contraction if and
		only if the map $f$ is a $\phi$-contraction such that the map $\phi$ is strictly increasing and concave.
	\end{remark}
	
	\begin{example}
		Let $X=\mathbb{R}_+$ and $d$ be the Euclidean metric defined on $X$. Let $f:X\rightarrow X$ be a map defined on $X$ such that $f(x)=\frac{1}{1+x}$. Then, notice that for $\phi : \mathbb{R}_+\rightarrow \mathbb{R}_+$ defined as $\phi(t)=\frac{t}{1+t}$, we have
		\begin{align*}
			d(f(x),f(y))=& \left\lvert \frac{1}{1+x} -\frac{1}{1+y} \right\rvert
			= \left \lvert \frac{x-y}{1+x+y+xy} \right\rvert
			\leq  \frac{\lvert x-y \rvert}{1+\lvert x-y \rvert}=\phi (d(x,y)).
		\end{align*}
		Hence, $f$ is a $\phi$-contraction map. Now, it is easy to observe that $\phi$ is a non-decreasing function and $\phi^{n}(t)\to 0$ as $n\to \infty$ for $t>0$. Hence, $f$ is a Matkowski contraction.
	\end{example}

	\begin{example}
		Let $X=\mathbb{R}_+$ endowed with the Euclidean metric. Define $f:X \rightarrow X$ by $f(x)=\frac{x^2}{1+x}$. Then, note that for function $\phi: \mathbb{R}_{+}\rightarrow \mathbb{R}_{+}$, defined as $\phi(x)=\frac{x^2}{1+x}$, we have
		\begin{align*}
			d(f(x),f(y))=& \left\lvert \frac{x^2}{1+x} -\frac{y^2}{1+y} \right\rvert
			= \frac{\lvert x-y\rvert^2 }{1+\lvert x-y \rvert}=\phi (d(x,y)).
		\end{align*}
		Hence, $f$ is a $\phi$-contraction. Moreover, it is easy to verify that $\phi$ is a non-decreasing function such that $\phi(t)<t$ for all $t>0$ and $\phi^{n}(t)\to 0$ as $n\to\infty$. Hence, $f$ is a Matkowski contraction. However, notice that the function $\beta(t)=\frac{\phi(t)}{t}=\frac{t}{1+t}$ is an increasing function. Hence, $f$ is not a Rakotch contraction.
	\end{example}
	\begin{definition}\label{jensendefinition}\cite{abramovich2004refining}
		Jensen’s inequality states that if $\phi : [0,\infty) \rightarrow \mathbb{R}$ is a convex function, then
		\begin{equation}\label{jensen}
			\phi\left(\int f  d\mu\right) \le \int \phi \left(f(s)\right) d\mu(s),
		\end{equation}
		for every probability measure $\mu$ and every nonnegative, $\mu$-integrable function $f : X \to [0,\infty)$.  If $\phi$ is concave, then the inequality in \eqref{jensen} is reversed, i.e., 
		\[ \phi\left(\int f  d\mu\right) \geq \int \phi \left(f(s)\right) d\mu(s).\]
	\end{definition}	
	
	\begin{theorem}\emph{\cite{strobin2015attractors}}\label{attractor}
		Let $Y$ be a complete metric space and $\{Y; \omega_1,\ldots, \omega_{\mathcal{N}}\}$ be an IFS satisfying Matkowski contractions. Then, there exists a unique non-empty compact set $G\in \mathcal{K}(Y)$ such that 
		\[G=\bigcup_{i=1}^{\mathcal{N}}\omega_{i}(G).\]
		Further, for any $\mathcal{A}\in \mathcal{K}(X)$, $\mathcal{W}^n(\mathcal{A})$ converges to $G$ as $n\to \infty$, with respect to the Hausdorff metric $\mathcal{H}_d$, where $\mathcal{W}$ is defined as follows
		\[\mathcal{W}(\mathcal{A})=\bigcup_{i=1}^{\mathcal{N}}\omega_{i}(\mathcal{A}).\]
	\end{theorem}

	\section{Fractal Interpolation Functions}\label{sec2}
	In this section, we discuss the multivariate FIFs from a new perspective.\\
	It is known that $\left(\mathcal{C}(I^q), d_{\mathcal{C}}\right)$, where \[d_{\mathcal{C}}(g,h)= \max_{\substack{(x_1,\ldots,x_q)\in I^q}}\left\lvert g(x_1,\ldots, x_q)-h(x_1,\ldots, x_q)\right\rvert \text{ for all } g,h \in \mathcal{C}(I^q),\] is a complete metric space.\\ 
	
	\noindent
	Let $N_1,\ldots, N_q$ be the positive integers greater than 1. The data of interpolation is given as
	\[\left\{\left(x_{1,i_1},\ldots, x_{q,i_q}, z_{(1,i_1),\ldots,(q,i_q)}\right): i_1=0,1,\ldots, N_1; ~ \ldots ~; i_q=0,1,\ldots, N_q\right\},\]
	such that $0=x_{k,0}<\cdots < x_{k,N_k}=1$ for each $I_k$, where $I_k$ denotes the $k^{th}$ unit interval in $I^q$.\\
	
	\noindent
	For a natural number, $N$ let us denote the following:
	\begin{align*}
		&\Sigma_N=\{1,\ldots,N\},~~ \Sigma_{N,0}=\{0,\ldots, N \},~~ \partial \Sigma_{N,0}=\{0,N\}, \text{ and } \text{int} \Sigma_{N,0}=\{1,\ldots,N-1\}.
	\end{align*}
	Further, a net $\Delta$ on $I^q$ is defined as follows:
	\[\Delta:=\left\{(x_{1,i_1}
	, \ldots , x_{q,i_q}) \in I ^q: i_k \in \Sigma_{N_k,0}, ~ 0= x_{k,0} <\ldots < x_{k,N_k} =1,~ k \in \Sigma_{q}\right\}.\]
	For each $i_k \in  \Sigma_{M_{k}}$, let us consider $I_{k,i_{k}} = [x_{k,i_{k}-1}, x_{k,i_{k}}]$ and define contractive homeomorphisms,  $u_{k,i_{k}}: I_k \rightarrow I_{k,i_k}$ such that
	\begin{equation}\label{affine}
		\begin{aligned}
			u_{k,i_k}(x_{k,0})=&x_{k,i_k-1}, \quad \quad u_{k,i_k}(x_{k,N_k})=x_{k,i_k}, \text{ and }\\
			\left\lvert u_{k,i_k}(x)-u_{k,i_k}(y) \right\rvert &\leq \mu_{k,i_k}\lvert x-y\rvert \text{ for all } x,y \in I_k, \text{ and } 0< \mu_{k,i_k}<1.
		\end{aligned}
	\end{equation}
	Define a function $\eta : \mathbb{Z}\times \{0, N_1, \ldots, N_q\}\rightarrow \mathbb{Z}$ by 
	\begin{equation*}
		\eta(i,0)=i-1, \quad \eta(i,N_1)=\cdots =\eta(i,N_q)=i.
	\end{equation*}
	
	\noindent
	Set $X=I^q\times \mathbb{R}$ and denote $\mathcal{D}_{0}$ as the Euclidean metric on $X$, i.e., for all $(x_1,\ldots, x_q, z^*)$ and  $(y_1,\ldots, y_q,z^{**})\in X$, we have
	\begin{align*}
		\mathcal{D}_0\left((x_1,\ldots,x_q,z^*),(y_1,\ldots,y_q,z^{**})\right)=&~\lVert (x_1,\ldots,x_q,z^*)-(y_1,\ldots,y_q,z^{**})\rVert_2\\
		=&~\sqrt{\lvert x_1-y_1\rvert ^2+\cdots+\lvert x_q-y_q\rvert^2+\lvert z^*-z^{**}\rvert^2}.
	\end{align*}
	For each $(i_1,\ldots,i_q)\in \prod_{k=1}^{q} \Sigma_{N_k}$, where $k\in \Sigma_q$, define $F_{i_1\cdots i_q}: X\rightarrow \mathbb{R}$ a continuous function such that 
	\begin{equation}\label{Fi}
		F_{i_1\cdots i_q}\left(x_{1,k_1},\ldots,x_{q,k_q},z_{(1,k_1)\cdots (q,k_q)}\right)=z_{\left(1,\eta(i_1,k_1)\right)\cdots \left(q,\eta(i_q,k_q)\right)}, 
	\end{equation}
	where $(k_1,\ldots k_q)\in \overset{q}{\prod_{\substack{i=1}}} \partial \Sigma_{N_i,0}$.
	Now define $\Omega_{i_1\cdots i_q}: X \rightarrow X$ such that 
	\begin{equation}\label{Wi}
		\Omega_{i_1\cdots i_q}(x_1,\ldots x_q, z)=\left(u_{1,i_1}(x_1),\ldots, u_{q,i_q}(x_q), F_{i_1\cdots i_q}(x_1,\ldots, x_q,z) \right),
	\end{equation}
	where $(i_1,\ldots,i_q)\in \prod_{k=1}^{q} \Sigma_{N_k}$.
	
	\noindent
	Denote $\mathcal{C}^*(I^q)$ and $\mathcal{C}^{**}(I^q)$ as
	\begin{align*}
		& \mathcal{C}^*(I^q) =\left\{f\in \mathcal{C}(I^q):~ f\left(x_{1,k_1},\ldots, x_{q,k_q}\right)=z_{(1,k_1)\cdots (q,i_q)} \text{ for } (k_1,\ldots, k_q)\in \overset{q}{\prod_{\substack{j=1}}} \partial \Sigma_{N_j,0}\right\}\\
		& \text{ and }\\
		& \mathcal{C}^{**}(I^q)=\Bigg\{f\in \mathcal{C}(I^q):~f\left(x_{1,i_1},\ldots, x_{q,i_q}\right)=z_{(1,i_1)\cdots (q,i_q)} \text{ for all } \left(x_{1,i_1},\ldots, x_{q,i_q}\right)\in \Delta \Bigg\},
	\end{align*}
	respectively. Note that $\mathcal{C}^*(I^q)$ and $\mathcal{C}^{**}(I^q)$ are closed subspaces of $\mathcal{C}(I^q)$, hence complete with respect to the metric $d_{\mathcal{C}}$.\\
	
	\noindent
	For all $f\in \mathcal{C}^*(I^q)$, define $T: \mathcal{C}^*(I^q)\rightarrow \mathcal{C}(I^q)$ such that
	\begin{equation}\label{Tf}
		Tf(x_1,\ldots, x_q) = F_{i_1\cdots i_q}\left(u_{1,i_1}^{-1}(x_1), \ldots, u_{q,i_q}^{-1}(x_q), f \left(u_{1,i_1}^{-1}(x_1), \ldots, u_{q,i_q}^{-1}(x_q)\right)\right)
	\end{equation}
	for $(x_1,\ldots, x_q)\in \prod_{k=1}^{q}I_{k,i_k}$, where $i_k\in \Sigma_{N_k}$.
	
	\begin{lemma}\label{operator}
		For all $f\in \mathcal{C}^*(I^q)$, we get $Tf\in \mathcal{C}^{**}(I^q)$, where $T$ is defined in \eqref{Tf}. This implies that $T: \mathcal{C}^{*}(I^q)\rightarrow \mathcal{C}^{**}(I^q)$ and $T^n: \mathcal{C}^{**}(I^q) \rightarrow \mathcal{C}^{**}(I^q)$ for $n\geq 2$.
	\end{lemma}
	\begin{proof}
		From \eqref{Tf}, it is easy to observe the well-definedness of $T$. Now to see that for all $f\in \mathcal{C}^*(I^q)$, $Tf\in \mathcal{C}^{**}(I^q)$.
		Let $f\in \mathcal{C}^*(I^q)$ and $(i_1,\ldots, i_q)\in \prod_{k=1}^{q}\text{int}\Sigma_{N_k,0}$. 
		Then, three cases may arise:\\
		\textbf{Case 1}. If $\left(x_{1,i_1},\ldots,x_{q,i_q}\right)\in \prod_{k=1}^{q}I_{k,i_k}$, then by \eqref{Tf}, we have
		\begin{align*}
			Tf\left(x_{1,i_1},\ldots,x_{q,i_q}\right)=&F_{i_1\cdots i_q}\left(u_{1,i_1}^{-1}(x_{1,i_1})\ldots, u_{q,i_q}^{-1}(x_{q,i_q}),f\left(u_{1,i_1}^{-1}(x_{1,i_1})\ldots, u_{q,i_q}^{-1}(x_{q,i_q})\right)\right)\\
			=& F_{i_1\cdots i_q}\left(x_{1,N_1},\ldots, x_{q,N_q}, f\left(x_{1,N_1},\ldots, x_{q,N_q}\right)\right)\\
			=& F_{i_1\cdots i_q}\left(x_{1,N_1},\ldots,x_{q,N_q},z_{(1,N_1)\cdots (q,N_q)}\right)\\
			=& z_{\left(1,\eta(i_1,N_1)\right) \cdots \left(q,\eta(i_q,N_q)\right)} =    z_{(1,i_1)\cdots (q,i_q)}.
		\end{align*}
		\textbf{Case 2}. If $\left(x_{1,i_1},\ldots,x_{q,i_q}\right)\in \prod_{k=1}^{q}I_{k,i_{k+1}}$, then we have
		\begin{align*}
			&Tf\left(x_{1,i_1},\ldots,x_{q,i_q}\right)\\
			=&F_{i_1+1\cdots i_q+1}\left(u_{1,i_1}^{-1}(x_{1,i_1})\ldots, u_{q,i_q}^{-1}(x_{q,i_q}),f\left(u_{1,i_1}^{-1}(x_{1,i_1})\ldots, u_{q,i_q}^{-1}(x_{q,i_q})\right)\right)\\
			=& F_{i_1+1\cdots i_q+1}\left(x_{1,0},\ldots, x_{q,0},f\left(x_{1,0},\ldots,x_{q,0}\right)\right)\\
			=& F_{i_1+1\cdots i_q+1}\left(x_{1,0}\ldots, x_{q,0},z_{(1,0)\cdots (q,0)}\right)\\
			=&z_{(1,\eta(i_1+1,0))\cdots (q,\eta(i_q+1,0))}=z_{(1,i_1)\cdots (q,i_q)}.
		\end{align*}
		\textbf{Case 3}. If $\left(x_{1,i_1},\ldots,x_{q,i_q}\right)\in \prod_{k=1}^{q}I_{k,i_{k}^{m}}$, where $i^{m}_{k}=i_k$ for some $k$ and $i_{k}^{m}=i_k+1$ for others, then using similar arguments as above, we get
		\[Tf\left(x_{1,i_1},\ldots, x_{q,i_q}\right)=z_{(1,i_1)\cdots (q, i_q)}.\]
		Similarly, one can prove that $Tf\left(x_{1,i_1},\ldots,x_{q,i_q}\right)=z_{(1,i_1)\cdots (q,i_q)}$ for all $(i_1,\ldots,i_q)\in \prod_{k=1}^{q}\partial \Sigma_{N_k,0}$. Hence, $Tf\left(x_{1,i_1},\ldots,x_{q,i_q}\right)=z_{(1,i_1)\cdots (q,i_q)}$ for all $\left(x_{1,i_1},\ldots,x_{q,i_q}\right)\in \Delta$, i.e., for all $f\in \mathcal{C}^{*}(I^q)$, we have $Tf\in \mathcal{C}^{**}(I^q)$.\\
		By \eqref{Tf}, $Tf$ is continuous on the $\prod_{k=1}^{q}I_{k,i_k}$ for all $(i_1,\ldots,i_q)\in \prod_{k=1}^{q}\Sigma_{N_k}$. Hence, $T^n :\mathcal{C}^{**}(I^q)\rightarrow \mathcal{C}^{**}(I^q)$ for all $n\geq 2$.
	\end{proof}
	\noindent
	\begin{theorem}\label{Matkowski}
		For all $(i_1,\ldots,i_q)\in \prod_{k=1}^{q}\Sigma_{N_k,0}$, let $F_{i_1\cdots i_q}$ be Matkowski contractions (with same function $\phi$) with respect to the $(q+1)^{th}$ variable, i.e., for some non-decreasing function $\phi: \mathbb{R}_{+}\rightarrow \mathbb{R}_{+}$ with $\phi^{n}(t)\to 0$ for $t>0$, each map, $F_{i_1\cdots i_q}$ satisfies :
		\[\left\lvert F_{i_1\cdots i_q}(x_1,\ldots, x_q,z)-F_{i_1\cdots i_q}(x_1,\ldots,x_q,z^*) \right\rvert\leq \phi \left(\lvert z-z^*\rvert\right)\]
		$\text{for all } (x_1,\ldots,x_q)\in I^q \text{ and } z,z^*\in \mathbb{R}.$ Then, the operator $T$ is a Matkowski contraction. Hence, there is a unique continuous function $f^*:I^q\to \mathbb{R}$ which is a fixed point of $T$. In particular, $f^*\left(x_{1,i_1},\ldots,x_{q,i_q}\right)=z_{(1,i_1)\cdots (q,i_q)}$. Moreover, the graph $\mathcal{G}$ of $f^*$ is invariant with respect to $\left\{X, \Omega_{i_1\cdots i_q}:~ (i_1,\ldots,i_q)\in \prod_{k=1}^{q}\Sigma_{N_k}\right\}$, i.e.,
		\[\mathcal{G}=\bigcup_{i_1=1}^{N_1}\cdots \bigcup_{i_q=1}^{N_q}\Omega_{i_1\cdots i_q}(\mathcal{G}).\]
	\end{theorem}
	
	\begin{proof}
		Since for all $f\in \mathcal{C}^*(I^q)$, $Tf\in \mathcal{C}^{**}(I^q)$, therefore we have
		{\footnotesize{
				\begin{align}
					&d_{\mathcal{C}}(Tg, Th)\nonumber\\
					=& \sup_{(x_1,\ldots,x_q)\in I^q} \lvert Tg(x_1,\ldots,x_q)-Th(x_1,\ldots,x_q)\rvert\nonumber\\
					=& \max_{(i_1,\ldots,i_q)\in \overset{q}{\underset{k=1}{\prod}}\Sigma_{N_k}}\sup_{(x_1,\ldots,x_q)\in \overset{q}{\underset{k=1}{\prod}}I_{k,i_k}}\left\lvert F_{i_1\cdots i_q}\left(u_{i,i_1}^{-1}(x_1),\ldots,u_{q,i_q}^{-1}(x_q),g\left(u_{i,i_1}^{-1}(x_1),\ldots,u_{q,i_q}^{-1}(x_q)\right)\right)\right.\nonumber\\
					&\hspace{4.5cm}\left.-F_{i_1\cdots i_q}\left(u_{i,i_1}^{-1}(x_1),\ldots,u_{q,i_q}^{-1}(x_q),h\left(u_{i,i_1}^{-1}(x_1),\ldots,u_{q,i_q}^{-1}(x_q)\right)\right) \right\rvert\nonumber\\
					\leq&  \max_{(i_1,\ldots,i_q)\in \overset{q}{\underset{k=1}{\prod}}\Sigma_{N_k}}\sup_{(x_1,\ldots,x_q)\in \overset{q}{\underset{k=1}{\prod}}I_{k,i_k}} \phi \left\lvert g\left(u_{i,i_1}^{-1}(x_1),\ldots,u_{q,i_q}^{-1}(x_q)\right)-h\left(u_{i,i_1}^{-1}(x_1),\ldots,u_{q,i_q}^{-1}(x_q)\right)\right\rvert. \label{eq3.5}
		\end{align}}}
		Since $\phi:[0,\infty)\rightarrow [0,\infty)$ is a non-decreasing function and $u_{k,i_k}^{-1}:\left[x_{k,i_k-1},x_{k,i_k}\right]\rightarrow \left[x_{k,0},x_{k,N_k}\right]$ for all $i_k\in \Sigma_{N_k}$, where $k\in \Sigma_{q}$, then for all $\left(x_{1,*},\ldots, x_{q,*}\right)\in \prod_{k=1}^{q}I_{k,i_{k0}}$ for all $\left(i_{10},\ldots,i_{q0}\right)\in \prod_{k=1}^{q}\Sigma_{N_k}$, we have
		{\footnotesize{
				\begin{align*}
					& \phi\left(\left\lvert g\left(u_{1,i_{10}}^{-1}(x_{1,*}),\ldots, u_{q,i_{q0}}^{-1}(x_{q,*})\right)-h\left(u_{1,i_{10}}^{-1}(x_{1,*}),\ldots, u_{q,i_{q0}}^{-1}(x_{q,*})\right) \right\rvert\right)\\
					\leq & \phi \left(\max_{(x_1,\ldots,x_q)\in \overset{q}{\underset{k=1}{\prod}}I_{k,i_k}}\left\lvert g\left(u_{1,i_{10}}^{-1}(x_{1,*}),\ldots, u_{q,i_{q0}}^{-1}(x_{q,*})\right)-h\left(u_{1,i_{10}}^{-1}(x_{1,*}),\ldots, u_{q,i_{q0}}^{-1}(x_{q,*})\right)\right\rvert\right)\\
					\leq & \phi \left(\max_{(x_1,\ldots,x_q)\in I^q}\left\lvert g\left(u_{1,i_{10}}^{-1}(x_{1,*}),\ldots, u_{q,i_{q0}}^{-1}(x_{q,*})\right)-h\left(u_{1,i_{10}}^{-1}(x_{1,*}),\ldots, u_{q,i_{q0}}^{-1}(x_{q,*})\right)\right\rvert\right)
					=\phi\left(d_{\mathcal{C}}(g,h)\right).
		\end{align*}}}
		Since $(x_{1,*},\ldots,x_{q,*})$ and $(i_{10},\ldots,i_{q0})$ are arbitrary, therefore we have
		{\footnotesize{
				\begin{align}
					&\sup_{(x_1,\ldots, x_q)\in \overset{q}{\underset{k=1}{\prod}}I_{k,i_{k0}}}\phi\left(\left\lvert g\left(u_{1,i_{10}}^{-1}(x_{1}),\ldots, u_{q,i_{q0}}^{-1}(x_{q})\right)-h\left(u_{1,i_{10}}^{-1}(x_{1}),\ldots, u_{q,i_{q0}}^{-1}(x_{q})\right) \right\rvert\right)\nonumber\\
					\leq& \max_{(i_1,\ldots,i_q)\in \overset{q}{\underset{k=1}{\prod}}\Sigma_{N_k}}\sup_{(x_1,\ldots, x_q)\in \overset{q}{\underset{k=1}{\prod}}I_{k,i_{k0}}}\phi\left(\left\lvert g\left(u_{1,i_{10}}^{-1}(x_{1}),\ldots, u_{q,i_{q0}}^{-1}(x_{q})\right)-h\left(u_{1,i_{10}}^{-1}(x_{1}),\ldots, u_{q,i_{q0}}^{-1}(x_{q})\right) \right\rvert\right)\nonumber\\
					\leq & \phi\left(d_{\mathcal{C}}(g,h)\right) \label{eq3.6}.
		\end{align}}}
		
		\noindent
		Using \eqref{eq3.5} and \eqref{eq3.6}, we get 
		\[d_{\mathcal{C}}(Tg, Th)\leq \phi\left(d_{\mathcal{C}}(g,h)\right).\]
		This proves that $T$ is a Matkowski $\phi$-contraction map on $\left(\mathcal{C}^*(I^q),d_{\mathcal{C}}\right)$. Hence, $T$ has a fixed point, say $f^*$ in $\mathcal{C}^*(I^q)$, i.e., $Tf^*(x_1,\ldots,x_q)=f^*(x_1,\ldots,x_q)$ for all $(x_1,\ldots,x_q)\in \mathcal{C}^*(I^q)$. Since $T:\mathcal{C}^{*}(I^q)\rightarrow \mathcal{C}^{**}(I^q)$, we have $f^*(x_1,\ldots,x_q)=Tf^*(x_1,\ldots,x_q)\in \mathcal{C}^{**}(I^q).$ Therefore, 
		\[f^*\left(x_{1,i_1},\ldots, x_{q,i_q}\right)=z_{(1,i_1)\cdots(q,i_q)} \text{ for all } (i_1,\ldots,i_q)\in \prod_{k=1}^{q}\Sigma_{N_k,0}.\]
		Since $f^*$ is a fixed point of $T$, then for all $(x_1,\ldots, x_q)\in I^q$, we have
		\begin{align*}
			f^*\left(u_{1,i_1}^{-1}(x_1),\ldots, u_{q,i_q}^{-1}(x_q)\right)= & Tf^*\left(u_{1,i_1}^{-1}(x_1),\ldots, u_{q,i_q}^{-1}(x_q)\right)\\
			&=F_{i_1\cdots i_q}\left(x_1,\ldots,x_q,f^*(x_1,\ldots, x_q)\right).
		\end{align*}
		Now since
		\[\Omega_{i_1\cdots i_q}(x_1,\ldots,x_q,z^*)=\left(u_{1,i_1}(x_1),\ldots,u_{q,i_q}(x_q), F_{i_1\cdots i_q}(x_1,\ldots,x_q,z^*)\right)\]
		for all $(i_1,\ldots,i_q)\in \prod_{k=1}^{q}\Sigma_{N_k}$. Then, we have
		\begin{align*}
			&\Omega_{i_1\cdots i_q}\left(\mathcal{G}(f^*)\right)\\
			=&\left\{\Omega_{i_1\cdots i_q}\left(x_1,\ldots,x_q,f^*(x_1,\ldots, x_q)\right): (x_1,\ldots, x_q)\in I^q\right\}\\
			=&\left\{\left(u_{1,i_1}(x_1),\ldots,u_{q,i_q}(x_q), F_{i_1\cdots i_q}\left(x_1,\ldots,x_q,f^*(x_1,\ldots, x_q)\right)\right): (x_1,\ldots,x_q)\in I^q\right\}\\
			=& \left\{\left(x_1,\ldots,x_q,f^*(x_1,\ldots,x_q)\right): (x_1,\ldots, x_q)\in \prod_{k=1}^{q}I_{k,i_k}\right\}.
		\end{align*}
		Hence,
		\begin{align*}
			\mathcal{G}(f^*)=&\left\{(x_1,\ldots, x_q, f^*(x_1,\ldots,x_q)): (x_1,\ldots,x_q)\in I^q\right\}\\
			=&\bigcup_{i_1=1}^{N_1}\cdots \bigcup_{i_q=1}^{N^q}\left\{\left(x_1,\ldots,x_q,f^*(x_1,\ldots,x_q)\right): (x_1,\ldots, x_q)\in \prod_{k=1}^{q}I_{k,i_k}\right\}\\
			=& \bigcup_{i_1=1}^{N_1}\cdots \bigcup_{i_q=1}^{N^q} \Omega_{i_1\cdots i_q}(\mathcal{G}(f^*)).
		\end{align*}
		This completes the proof.
	\end{proof}
	\begin{theorem}\label{rakotchthoerem}
		For any $t>0$, let $\alpha(t)=\frac{\phi(t)}{t}<1$ and the function $\alpha$ be a non-increasing function (as the function $\phi$) with respect to $(q+1)^{th}$ variable and Lipschitz with respect to the rest of the variables, i.e., for some $\mathcal{M}_1,\ldots,  \mathcal{M}_q$ and some non-decreasing function $\phi: \mathbb{R}_+\rightarrow \mathbb{R}_+$ with $\phi(t)<t$ for $t>0$ such that the map $t\to \frac{\phi(t)}{t}$ is non-increasing, and
		{\footnotesize{
				\begin{equation}\label{rakotch}
					\left\lvert F_{i_1\cdots i_q}(x_1,\ldots,x_q,z^*)-F_{i_1\cdots i_q}(y_1,\ldots, y_q,z^{**})\right\rvert \leq\mathcal{M}_1\lvert x_1-y_1 \rvert+\ldots + \mathcal{M}_{q}\lvert x_q -y_q \rvert + \phi \left(\lvert z^*-z^{**}\rvert\right) 
				\end{equation}
		}}
		for all $(x_1,\ldots, x_q), (y_1,\ldots,y_q)\in I^q,$ and  $z^*,z^{**}\in \mathbb{R}$. Then, there is a metric $\mathcal{D}_{\theta}$ on $X$, equivalent to the Euclidean metric $\mathcal{D}_{0}$ such that for all $(i_1,\ldots,i_q)\in \prod_{k=1}^{q}\Sigma_{N_k}$, $\Omega_{i_1\cdots i_q}$ are Rakotch contraction maps with respect to $\mathcal{D}_{\theta}$.
		In particular, there exists a unique non-empty compact set $\mathcal{G}\subset X=I^q\times \mathbb{R}$ such that 
		\[\mathcal{G}=\bigcup_{i_1=1}^{N_1}\cdots \bigcup_{i_q=1}^{N_q}\Omega_{i_1\cdots i_q}(\mathcal{G}).\]
		Moreover, $\mathcal{G}$ is the graph of a continuous function $f^*: I^q\rightarrow \mathbb{R}$ which is a fixed point of the operator $T$ defined in \eqref{Tf}.
	\end{theorem}
	\begin{proof}
		Define a map $\mathcal{D}_{\theta}: X\rightarrow \mathbb{R}$ such that
		\[\mathcal{D}_{\theta}\left((x_1,\ldots,x_q,z^*),(y_1,\ldots,y_q,z^{**})\right)=\lvert x_1-y_1\rvert+\cdots +\lvert x_q-y_q \rvert +\theta \lvert z^*-z^{**}\rvert\]
		for all $(x_1,\ldots,x_q,z^*), (y_1,\ldots,y_q,z^{**})\in X,$ where $\theta$ is a positive real number. It is easy to note that $\mathcal{D}_{\theta}$ is a metric defined on $X$ and it is equivalent to the Euclidean metric $\mathcal{D}_{0}$ defined on $X$. Then, for all $(x_1,\ldots,x_q,z^*),(y_1,\ldots,y_q,z^{**})\in X$, we have
		{\footnotesize{
				\begin{align*}
					&~\mathcal{D}_{\theta}\left(\Omega_{i_1\cdots i_q}(x_1,\ldots,x_q,z^*),\Omega_{i_1\cdots i_q}(y_1,\ldots,y_q,z^{**})\right)\\
					=&~ \mathcal{D}_{\theta} \left(\left(u_{1,i_1}(x_1),\ldots, u_{q,i_q}(x_q), F_{i_1\cdots i_q}(x_1,\ldots,x_q,z^{*})\right), \left(u_{1,i_1}(y_1),\ldots,u_{q,i_q}(y_q), F_{i_1\cdots i_q}(y_1,\ldots,y_q,z^{**})\right)\right)\\
					=&~ \left\lvert u_{1,i_1}(x_1)-u_{1,i_1}(y_1) \right\rvert +\cdots + \left\lvert u_{q,i_q}(x_q)-u_{q,i_q}(y_q) \right\rvert + \left \lvert F_{i_1\cdots i_q}(x_1,\ldots,x_q,z^{*})- F_{i_1\cdots i_q}(y_1,\ldots,y_q,z^{**})\right\rvert\\
					=&~\lvert \mu_{1,i_1}\rvert \lvert x_1-y_1 \rvert +\cdots + \lvert \mu_{q,i_q}\rvert \lvert x_q-y_q\rvert+ \theta \left(\mathcal{M}_1\lvert x_1-y_1\rvert +\cdots + \mathcal{M}_q\lvert x_q-y_q \rvert +\phi \left(z^*-z^{**}\right)\right)\\
					=&~\left(\lvert \mu_{1,i_1}\rvert +\theta \mathcal{M}_{1}\right)\lvert x_1-y_1\rvert+\cdots +\left(\lvert \mu_{q,i_q}\rvert +\theta \mathcal{M}_{q}\right)\lvert x_q-y_q \rvert +\theta \phi \left(\lvert z^*-z^{**}\rvert\right).
		\end{align*}}}
		
		\noindent
		Since $\phi :(0,\infty)\rightarrow (0,\infty)$ is a non-decreasing function and $\phi(t)<t$ for all $t>0$, then for $(x_1,\ldots,x_q,z^*)\neq (y_1,\ldots,y_q,z^{**})$, we have
		{\footnotesize{
				\begin{align*}
					&~ \mathcal{D}_{\theta}\left(\Omega_{i_1\cdots i_q}(x_1,\ldots,x_q,z^*), \Omega_{i_1\cdots i_q}(y_1,\ldots,y_q,z^{**})\right)\\
					\leq &~ \left(\lvert \mu_{1,i_1}\rvert +\theta \mathcal{M}_1\right)\lvert x_1-y_1\rvert +\cdots +\left(\lvert \mu_{q,i_q}\rvert +\theta \mathcal{M}_q\right)\lvert x_q-y_q\rvert+ \theta \tfrac{\phi(\lvert z^*-z^{**} \rvert)}{\lvert x_1-y_1\rvert +\cdots +\lvert x_q-y_q\rvert+\lvert z^*-z^{**}\rvert}\\
					&\hspace{8cm}\left(\lvert x_1-y_1\rvert+\cdots +\lvert x_q-y_q\rvert+\lvert z^*-z^{**}\rvert\right)\\
					\leq &~ \left(\lvert \mu_{1,i_1}\rvert +\theta \mathcal{M}_{1}+\theta \tfrac{\phi(\lvert x_1-y_1\rvert +\cdots +\lvert x_q-y_q\rvert+\lvert z^*-z^{**}\rvert)}{\lvert x_1-y_1\rvert +\cdots +\lvert x_q-y_q\rvert+\lvert z^*-z^{**}\rvert}\right)\lvert x_1-y_1 \rvert+\cdots+ \Big(\lvert \mu_{q,i_q}\rvert +\theta \mathcal{M}_{q}\\
					&~\hspace{2cm}\left.+\theta \tfrac{\phi(\lvert x_1-y_1\rvert +\cdots +\lvert x_q-y_q\rvert+\lvert z^*-z^{**}\rvert)}{\lvert x_1-y_1\rvert +\cdots +\lvert x_q-y_q\rvert+\lvert z^*-z^{**}\rvert}\right)\lvert x_q-y_q \rvert \theta \tfrac{\phi(\lvert x_1-y_1\rvert +\cdots +\lvert x_q-y_q\rvert+\lvert z^*-z^{**}\rvert)}{\lvert x_1-y_1\rvert +\cdots +\lvert x_q-y_q\rvert+\lvert z^*-z^{**}\rvert}\lvert z^*-z^{**} \rvert\\
					\leq&~ \left(\lvert \mu_{1,i_1}\rvert +\theta \mathcal{M}_{1}+\theta\right)\lvert x_1-y_1 \rvert +\cdots + \left(\lvert \mu_{q,i_q}\rvert +\theta \mathcal{M}_{q}+\theta\right)\lvert x_q-y_q \rvert + \theta \tfrac{\phi(\lvert x_1-y_1\rvert +\cdots +\lvert x_q-y_q\rvert+\lvert z^*-z^{**}\rvert)}{\lvert x_1-y_1\rvert +\cdots +\lvert x_q-y_q\rvert+\lvert z^*-z^{**}\rvert}\lvert z^*-z^{**} \rvert\\
					\leq &~ \max \left\{(\lvert \mu_{1,i_1}\rvert+\theta \mathcal{M}_{1}+\theta), \ldots, (\lvert \mu_{q,i_q}\rvert+\theta \mathcal{M}_{q}+\theta),  (\tfrac{\phi(\lvert x_1-y_1\rvert +\cdots +\lvert x_q-y_q\rvert+\lvert z^*-z^{**}\rvert)}{\lvert x_1-y_1\rvert +\cdots +\lvert x_q-y_q\rvert+\lvert z^*-z^{**}\rvert})\right\}\\
					&\hspace{7cm}\left(\lvert x_1-y_1\rvert +\cdots +\lvert x_q-y_q\rvert +\theta \lvert z^*-z^{**}\rvert\right).
		\end{align*}}}
		
		\noindent
		Let us take $\theta=\min_{j\in \Sigma_q}\left\{\frac{1-\max_{i_j\in \Sigma_{N_j}}\lvert \mu_{j,i_j}\rvert}{2(\mathcal{M}_{j}+1)} \right\}$, then with the assumption that $0<\lvert \mu_{j,i_j}\rvert<1$ for all $i_j\in \Sigma_{N_j}$, where $j\in \Sigma_q$, we have $0<\theta <1$ and $0<\max_{i_j\in \Sigma_{N_j}}\lvert \mu_{j,i_j}\rvert +\theta \mathcal{M}_j+\theta<1$.\\
		Since $t\to \frac{\phi(t)}{t}$ is a non-increasing function and $0<\theta<1$, we have
		{\footnotesize{
				\begin{align*}
					&~\mathcal{D_{\theta}}\left(\Omega_{i_1\cdots i_q}(x_1,\ldots,x_q,z^*), \Omega_{i_1\cdots i_q}(y_1,\ldots,y_q,z^{**})\right)\\
					\leq &~ \max \left\{(\lvert \mu_{1,i_1}\rvert +\theta\mathcal{M}_{1}+\theta),\ldots, (\lvert \mu_{q,i_q}\rvert+\theta \mathcal{M}_{q}+\theta), (\tfrac{\phi(\lvert x_1-y_1\rvert +\cdots +\lvert x_q-y_q\rvert+\lvert z^*-z^{**}\rvert)}{\lvert x_1-y_1\rvert +\cdots +\lvert x_q-y_q\rvert+\lvert z^*-z^{**}\rvert})\right\}\\
					&\hspace{7cm}\left(\lvert x_1-y_1\rvert +\cdots +\lvert x_q-y_q\rvert +\theta \lvert z^*-z^{**}\rvert\right)\\
					\leq &~ \max\left\{(\lvert \mu_{1,i_1}\rvert +\theta\mathcal{M}_{1}+\theta),\ldots, (\lvert \mu_{q,i_q}\rvert+\theta \mathcal{M}_{q}+\theta), (\tfrac{\phi(\lvert x_1-y_1\rvert +\cdots +\lvert x_q-y_q\rvert+\lvert z^*-z^{**}\rvert)}{\lvert x_1-y_1\rvert +\cdots +\lvert x_q-y_q\rvert+\lvert z^*-z^{**}\rvert})\right\}\\
					&\hspace{7cm}\mathcal{D}_{\theta}\left((x_1,\ldots,x_q,z^*),(y_1,\ldots,y_q,z^{**})\right)\\
					\leq&~\max\left\{(\lvert \mu_{1,i_1}\rvert +\theta\mathcal{M}_{1}+\theta),\ldots, (\lvert \mu_{q,i_q}\rvert+\theta \mathcal{M}_{q}+\theta), (\tfrac{\phi\left(\mathcal{D}_{\theta}\left((x_1,\ldots,x_q,z^*),(y_1,\ldots,y_q,z^{**})\right)\right)}{\mathcal{D}_{\theta}\left((x_1,\ldots,x_q,z^*),(y_1,\ldots,y_q,z^{**})\right)})\right\}\\
					&\hspace{7cm}\mathcal{D}_{\theta}\left((x_1,\ldots,x_q,z^*),(y_1,\ldots,y_q,z^{**})\right).
		\end{align*}}}
		
		\noindent
		Let $\beta(t)=\max\left\{(\lvert \mu_{1,i_1}\rvert +\theta\mathcal{M}_{1}+\theta),\ldots, (\lvert \mu_{q,i_q}\rvert+\theta \mathcal{M}_{q}+\theta), (\tfrac{\phi(t)}{t})\right\}$ for all $t>0$. Then, notice that $\beta: (0,\infty)\rightarrow [0,1)$ is a non-increasing function and for each $(x_1,\ldots,x_q,z^*),(y_1,\ldots,y_q,z^{**})\in X$, we have
		\begin{align*}
			&\mathcal{D_{\theta}}\left(\Omega_{i_1\cdots i_q}(x_1,\ldots,x_q,z^{*}),\Omega_{i_1\cdots i_q}(y_1,\ldots,y_q,z^{**})\right)\\
			\leq&~ \beta\left(\mathcal{D}_{\theta}(x_1,\ldots,x_q,z^*),(y_1,\ldots,y_q,z^{**})\right)\mathcal{D}_{\theta}\left((x_1,\ldots,x_q,z^*),(y_1,\ldots,y_q,z^{**})\right).
		\end{align*}
		Hence, for each $(i_1,\ldots,i_q)\in \prod_{k=1}^{q}\Sigma_{N_k}$, $\Omega_{i_1\cdots i_q}:X\rightarrow X$ is a Rakotch contraction map in $(X, \mathcal{D}_{\theta})$. Further, since $\mathcal{D}_{\theta}$ is equivalent to the Euclidean metric $\mathcal{D}_{0}$, hence $(X, \mathcal{D}_{\theta})$ is also a complete metric space. Therefore, by using Theorem \ref{attractor}, for the complete metric space $\left(X, \mathcal{D}_{\theta}\right)$ there is a unique non-empty compact set $\mathcal{G}\subset X$ such that 
		\[\mathcal{G}=\bigcup_{i_1=1}^{N_1}\cdots \bigcup_{i_q=1}^{N_q}\Omega_{i_1\cdots i_q}(\mathcal{G}),\]
		and for any $\mathcal{A}\in \mathcal{K}(X)$, $\mathcal{W}^{n}(\mathcal{A})$ converges to $\mathcal{G}\in \mathcal{K}(X)$ as $n\to \infty$ with respect to the Hausdorff metric $\mathcal{H}_{\mathcal{D}_{\theta}}$, where $\mathcal{W}(\mathcal{A})=\bigcup_{i_1=1}^{N_1}\cdots \bigcup_{i_q=1}^{N_q}\Omega_{i_1\cdots i_q}(\mathcal{A})$.
		
		\noindent
		Since $\mathcal{D}_{0}$ and $\mathcal{D}_{\theta}$ are equivalent, therefore $\mathcal{H}_{\mathcal{D}_{0}}$ and  $\mathcal{H}_{\mathcal{D}_{\theta}}$ are equivalent. Hence, for any $\mathcal{A}\in \mathcal{K}(X)$, $\mathcal{W}^n(\mathcal{A})$ converges to $\mathcal{G}\in \mathcal{K}(X)$ as $n\to \infty$ with respect to the Hausdorff metric $\mathcal{H}_{\mathcal{D}_0}$.
		
		\noindent
		Now by Theorem \ref{Matkowski}, $T$ has a fixed point $f^*$ such that 
		\[\mathcal{G}(f^*)=\bigcup_{i_1=1}^{N_1}\cdots \bigcup_{i_q=1}^{N_q}\Omega_{i_1\cdots i_q}(\mathcal{G}(f^*)).\]
		From the uniqueness of $\mathcal{G}$, it follows that $\mathcal{G}$ must be the $\mathcal{G}(f^*)$. This completes the proof.
	\end{proof}
	\section{Generalized Multivariate $\alpha$-Fractal Functions}\label{sec4}
	In this section, we discuss the multivariate $\alpha$-fractal functions constructed by using generalized IFS. Let $I^q,C(I^q),\Delta$ be defined as in Section~\ref{sec2}. For each $i_k \in \sum_{M_k}$, set $I_{k,i_k}=[x_{k,i_{k-1}},x_{k,i_k}]$. Let $\alpha_{i_k} = (\alpha_{1,i_1}, \cdots,\alpha_{q,i_q})$ be  a scaling vector such that $|\alpha_{i_k}| < 1$ for all $i_k\in \sum_{M_k}$. Let us consider a continuous function $b \in C^*(I^q)$ such that $b:I^q \rightarrow \mathbb{R}$ and $b \neq f$ satisfying  $b(x_{1,k_1},\cdots,x_{q,k_q})=f(x_{1,k_1},\cdots,x_{q,k_q}) $ for all $(k_1,\cdots,k_q) \in \prod_{j=1}^{q}\partial \sum_{N_j,0}$. The function $b$ is called the base function. \\
	\subsection{ Multivariate $\alpha$-Fractal Functions for Matkowski contraction}
	Let $K$ be a closed interval in $\mathbb{R}$ such that $f(I^q)\subset K$. For each $(i_1,\cdots,i_q) \in \prod_{k=1}^{q}  \sum_{N_k}$, where $k\in\sum_{q}$, we define a continuous mapping $F_{i_1,\cdots,i_q}: I^q \times K \rightarrow K$ by
	\begin{align*}
		F_{i_1,\cdots,i_q}((x_{1,k_1},\cdots,x_{q,k_q}),f(x_{1,k_1},\cdots,x_{q,k_q}))=& r_{i_1,\cdots,i_q}((x_{1,k_1},\cdots,x_{q,k_q}),f(x_{1,k_1},\cdots,x_{q,k_q}))\\&+f_{i_1,\cdots,i_q}(u_{1,i_1}(x_1),\cdots,u_{q,i_q}(x_q))\\
		&-r_{i_1,\cdots,i_q}((x_{1,k_1},\cdots,x_{q,k_q}),b(x_{1,k_1},\cdots,x_{q,k_q}),
	\end{align*}
	where $f_{i_1,\cdots,i_q}$ is a given continuous function, $r_{i_1,\cdots,i_q}: I^q \times K \rightarrow \mathbb{R}$ is a continuous map and $r_{i_1,\cdots,i_q}$ is a Matkowski contraction such that 
	\[|r_{i_1,\cdots,i_q}(x_1,\cdots,x_q , y)-r_{i_1,\cdots,i_q}(x_1,\cdots,x_q , y^*) | \leq \phi (|y-y^*|),\]
	where $\phi $ is a non-decreasing function and for every $t>0$, $\phi^{n}(t)\to 0$ as $n\to \infty$. For each  $(i_1,\cdots,i_q) \in \prod_{k=1}^{q}   \sum_{N_k}$, we define a map $W_{i_1,\cdots,i_q}: I^q \times K \rightarrow I^q \times K$  as follows
	\begin{align*}
		W_{i_1,\cdots,i_q}((x_{1,k_1},\cdots,x_{q,k_q}),&f(x_{1,k_1},\cdots,x_{q,k_q})) =\\& (u_{1,i_1}(x_1),\cdots,u_{q,i_q}(x_q)),F_{i_1,\cdots,i_q}((x_{1,k_1},\cdots,x_{q,k_q}),f(x_{1,k_1},\cdots,x_{q,k_q}))).
	\end{align*}
	Then, $\mathbb{I}=\{I^q \times K ;W_{i_1,\cdots,i_q} :(i_1,\cdots,i_q) \in \prod_{k=1}^{q}   \sum_{N_k} \}$ is an IFS.
	\begin{theorem}\label{alphaMatkowski}
		For each  $(i_1,\cdots,i_q) \in \prod_{k=1}^{q}\sum_{N_k}$, let $r_{i_1,\cdots,i_q}$ be Matkowski contraction and IFS  $\mathbb{I}$ be defined as above. Then, there exist a unique continuous function $f^{\phi }_{\Delta,b} :I^q \rightarrow K $ such that $f^{\phi} _{\Delta,b}(x_{1,k_1},\cdots,x_{q,k_q}  )=f(x_{1,k_1},\cdots,x_{q,k_q}) $ for each  $(i_1,\cdots,i_q) \in \prod_{k=1}^{q}   \sum_{N_k}$ and the function $f^{\phi} _{\Delta,b}$ satisfies the following self-referential equation
		\begin{align*}
			f^{\phi} _{\Delta,b}(u_{1,i_1}(x_1),\cdots,u_{q,i_q}(x_q))=& f_{i_1,\cdots,i_q}(u_{1,i_1}(x_1),\cdots,u_{1,i_1}(x_1))\\
			&+r_{i_1,\cdots,i_q}((x_{1,k_1},\cdots,x_{q,k_q}),f^{\phi} _{\Delta,b}(x_{1,k_1},\cdots,x_{q,k_q}))\\
			&-r_{i_1,\cdots,i_q}((x_{1,k_1},\cdots,x_{q,k_q}),b(x_{1,k_1},\cdots,x_{q,k_q}))
		\end{align*}
		for all $(x_{1,k_1},\cdots,x_{q,k_q})\in I^q$ and  $(i_1,\cdots,i_q) \in \prod_{k=1}^{q}   \sum_{N_k}$. Furthermore , the graph $G_{f^{\alpha} _{\Delta,b}}(I^q)$ of $f^{\alpha} _{\Delta,b}$ is the attractor of the IFS $\mathbb{I}$.
	\end{theorem}
	\begin{proof}
		First, we define a complete metric space
		\begin{align*}
			\mathcal{C}^*(I^q) =\bigg\{g :I^q \rightarrow K | \text{g is a continuous map and } g(x_{1,k_1},\ldots, x_{q,k_q})=f(x_{1,k_1},\ldots, x_{q,k_q}) \\ \text{ for } (k_1,\ldots, k_q)\in \overset{q}{\prod_{\substack{j=1}}} \partial \Sigma_{N_j,0}\bigg\}.
		\end{align*}
		It is easy to observe that $(\mathcal{C}^*(I^q),d_{\mathbb{C}})$ is a complete metric space. Now, we define the RB operator $\mathbb{T}:\mathcal{C}^*(I^q) \rightarrow \mathcal{C}^*(I^q) $ by 
		\begin{align*}
			\mathbb{T}g(x_1,\ldots,x_q)=&r_{i_1,\cdots,i_q}\left(u_{1,i_1}^{-1}(x_1), \ldots, u_{q,i_q}^{-1}(x_q), g \left(u_{1,i_1}^{-1}(x_1), \ldots, u_{q,i_q}^{-1}(x_q)\right)\right)+ f(x_1,\ldots, x_q)\\
			&-r_{i_1,\cdots,i_q}\left(\left(u_{1,i_1}^{-1}(x_1), \ldots, u_{q,i_q}^{-1}(x_q)\right),b\left(u_{1,i_1}^{-1}(x_1), \ldots, u_{q,i_q}^{-1}(x_q)\right)\right)
		\end{align*}
		for $(x_1,\ldots, x_q) \in \prod_{k=1}^{q}I_{k,i_k}$, where $i_k\in \sum_{N_k}$.
		
		It is not difficult to see that the map $\mathbb{T}$ is well defined and for each $g\in C^*(I^q)$, $\mathbb{T}g(x_{1,k_1},\ldots, x_{q,k_q})=f(x_{1,k_1},\ldots, x_{q,k_q})$ where $(k_1,\ldots,k_q) \in \prod_{i=1}^{q} \sum_{N_i}$. Now we will show that the operator $\mathbb{T}$ is a Matkowski contraction. Let $g,h \in C^*(I^q)$, then we have
		\begin{align*}
			d_c(\mathbb{T}g,\mathbb{T}h)=& \lVert\mathbb{T}g-\mathbb{T}h\rVert_C\\
			=&\sup_{(x_1,\ldots,x_q)\in I^q} \lvert \mathbb{T}g(x_1,\ldots, x_q)-\mathbb{T}h(x_1,\ldots, x_q)\rvert\\
			\leq&\max_{(i_1,\ldots,i_q)\in \overset{q}{\underset{k=1}{\prod}}\Sigma_{N_k}}\sup_{(x_1,\ldots,x_q)\in \overset{q}{\underset{k=1}{\prod}}I_{k,i_k}}\bigg\lvert r_{i_1,\cdots,i_q}\left(u_{1,i_1}^{-1}(x_1), \ldots, u_{q,i_q}^{-1}(x_q),\right. \\
			&\hspace{5.5cm}\left.g \left(u_{1,i_1}^{-1}(x_1), \ldots, u_{q,i_q}^{-1}(x_q)\right)\right) \\
			&- r_{i_1,\cdots,i_q}\left(u_{1,i_1}^{-1}(x_1), \ldots, u_{q,i_q}^{-1}(x_q), h \left(u_{1,i_1}^{-1}(x_1), \ldots, u_{q,i_q}^{-1}(x_q)\right)\right) \bigg\rvert.
		\end{align*}
		Therefore, 
		\begin{align*}
			d_c(\mathbb{T}g,\mathbb{T}h)\leq&\max_{(i_1,\ldots,i_q)\in \overset{q}{\underset{k=1}{\prod}}\Sigma_{N_k}} \phi(d_c(g,h))\\
			=& \Phi (d_c(g,h)),
		\end{align*}
		where $\Phi : \mathbb{R}_+ \rightarrow \mathbb{R}_+$ is given by 
		\[\Phi(t) = \max_{(i_1,\ldots,i_q)\in \overset{q}{\underset{k=1}{\prod}}\Sigma_{N_k}} \phi(t).\]
		Using the definition of $\phi(t)$, it is clear that  $\Phi$ is a non decreasing function and for every $t>0$, $\phi^{n}(t)\to 0$ as $n\to \infty$. This implies that $\mathbb{T}$ is a Matkowski contraction. Thus, by Theorem \ref{Matkowski} , there exist a unique function $f^{\phi} _{\Delta,b}\in \mathcal{C}^* (I^q)$ such that $\mathbb{T}f^{\phi} _{\Delta,b} = f^{\phi} _{\Delta,b}$ and it satisfies the following self-referential equation 
		\begin{align*}
			f^{\phi} _{\Delta,b}(u_{1,i_1}(x_1),\cdots,u_{q,i_q}(x_q))=& f_{i_1,\cdots,i_q}(u_{1,i_1}(x_1),\cdots,u_{q,i_q}(x_q))\\
			&+r_{i_1,\cdots,i_q}((x_{1,k_1},\cdots,x_{q,k_q}),f^{\phi} _{\Delta,b}(x_{1,k_1},\cdots,x_{q,k_q}))\\
			&-r_{i_1,\cdots,i_q}((x_{1,k_1},\cdots,x_{q,k_q}),b(x_{1,k_1},\cdots,x_{q,k_q}))
		\end{align*}
		for all $(x_{1,k_1},\cdots,x_{q,k_q})\in I^q$ and  $(i_1,\cdots,i_q) \in \prod_{k=1}^{q}   \sum_{N_k}$. Using above self-referential equation satisfied by$f^{\phi} _{\Delta,b}$, one can prove that 
		\[{\mathcal{G}_{f^{\alpha} _{\Delta,b}}(I^q)}=\bigcup_{i_1=1}^{N_1}\cdots \bigcup_{i_q=1}^{N_q} W_{i_1\cdots i_q}(\mathcal{G}_{f^{\alpha} _{\Delta,b}}(I^q)).\]
		Thus, the proof is complete.
	\end{proof}
	\subsection{ Multivariate $\alpha$-Fractal Functions for Rakotch Contraction}
	Let $K$ be a closed interval in $\mathbb{R}$ such that $f(I^q)\subset K$. For each $(i_1,\cdots,i_q) \in \prod_{k=1}^{q}  \sum_{N_k}$, where $k\in\sum_{q}$, we define a continuous mapping $F_{i_1,\cdots,i_q}: I^q \times K \rightarrow K$ by
	\begin{align*}
		F_{i_1,\cdots,i_q}((x_{1,k_1},\cdots,x_{q,k_q}),f(x_{1,k_1},\cdots,x_{q,k_q}))=& r_{i_1,\cdots,i_q}((x_{1,k_1},\cdots,x_{q,k_q}),f(x_{1,k_1},\cdots,x_{q,k_q}))\\&+f_{i_1,\cdots,i_q}(u_{1,i_1}(x_1),\cdots,u_{q,i_q}(x_q))\\&- r_{i_1,\cdots,i_q}((x_{1,k_1},\cdots,x_{q,k_q}),b(x_{1,k_1},\cdots,x_{q,k_q}),
	\end{align*}
	where $f_{i_1,\cdots,i_q}$ is a given continuous function, $r_{i_1,\cdots,i_q}: I^q \times K \rightarrow \mathbb{R}$ is a continuous map such that $r_{i_1,\cdots,i_q}$ is a Rakotch contraction satisfying 
	\[|r_{i_1,\cdots,i_q}(x_1,\cdots,x_q , y)-r_{i_1,\cdots,i_q}(x_1,\cdots,x_q , y^*) | \leq \phi (|y-y^*|)\]
	for any $t>0$, $\beta (t)=\frac{\phi(t)}{t}<1$  and the function $\beta$ is a non-increasing function. For each  $(i_1,\cdots,i_q) \in \prod_{k=1}^{q}   \sum_{N_k}$, we define a map $W_{i_1,\cdots,i_q}: I^q \times K \rightarrow I^q \times K$  as follows
	\begin{align*}
		&W_{i_1,\cdots,i_q}((x_{1,k_1},\cdots,x_{q,k_q}),f(x_{1,k_1},\cdots,x_{q,k_q}))\\ =&(u_{1,i_1}(x_1),\cdots,u_{q,i_q}(x_q)),F_{i_1,\cdots,i_q}((x_{1,k_1},\cdots,x_{q,k_q}),f(x_{1,k_1},\cdots,x_{q,k_q}))).
	\end{align*}
	Then, $\mathcal{I}=\{I^q \times K ;W_{i_1,\cdots,i_q} :(i_1,\cdots,i_q) \in \prod_{k=1}^{q}   \sum_{N_k} \}$ is an IFS
	\begin{theorem}\label{alphaRakotch}
		For each  $(i_1,\cdots,i_q) \in \prod_{k=1}^{q}   \sum_{N_k}$, let $r_{i_1,\cdots,i_q}$ be Rakotch contractions. Let the IFS  $\mathcal{I}$ be defined as above. Then there exist a unique continuous function $f^{\phi }_{\Delta,b} :I^q \rightarrow K $ such that $f^{\phi} _{\Delta,b}(x_{1,k_1},\cdots,x_{q,k_q}  )=f(x_{1,k_1},\cdots,x_{q,k_q}) $ for each  $(i_1,\cdots,i_q) \in \prod_{k=1}^{q}   \sum_{N_k}$ and the function $f^{\phi} _{\Delta,b}$ satisfies the following self-referential equation
		\begin{align*}
			f^{\phi} _{\Delta,b}(u_{1,i_1}(x_1),\cdots,u_{q,i_q}(x_q))=& f_{i_1,\cdots,i_q}(u_{1,i_1}(x_1),\cdots,u_{q,i_q}(x_q))\\
			&+r_{i_1,\cdots,i_q}((x_{1,k_1},\cdots,x_{q,k_q}),f^{\phi} _{\Delta,b}(x_{1,k_1},\cdots,x_{q,k_q}))\\
			&-r_{i_1,\cdots,i_q}((x_{1,k_1},\cdots,x_{q,k_q}),b(x_{1,k_1},\cdots,x_{q,k_q}))
		\end{align*}
		for all $(x_{1,k_1},\cdots,x_{q,k_q})\in I^q$ and  $(i_1,\cdots,i_q) \in \prod_{k=1}^{q}   \sum_{N_k}$. Furthermore , the graph of $G_{f^{\alpha} _{\Delta,b}}(I^q)$ is the attractor of the IFS $\mathcal{I}$.
	\end{theorem}
	\begin{proof}
		Define a map $\mathcal{D}_{\theta}: X\rightarrow \mathbb{R}$ such that
		\[\mathcal{D}_{\theta}\left((x_1,\ldots,x_q,z^*),(y_1,\ldots,y_q,z^{**})\right)=\lvert x_1-y_1\rvert+\cdots +\lvert x_q-y_q \rvert +\theta \lvert z^*-z^{**}\rvert\]
		for all $(x_1,\ldots,x_q,z^*), (y_1,\ldots,y_q,z^{**})\in X,$ where $\theta$ is a positive real number. It is easy to note that $\mathcal{D}_{\theta}$ is a metric defined on $X$ and it is equivalent to the Euclidean metric $\mathcal{D}_{0}$ defined on $X$. Then, for all $(x_1,\ldots,x_q,z^*),(y_1,\ldots,y_q,z^{**})\in X$, such that 
		{\footnotesize{
				\begin{align*}
					&~ \mathcal{D}_{\theta}\left(r_{i_1\cdots i_q}(x_1,\ldots,x_q,f(x_{1,k_1},\cdots,x_{q,k_q})^*), r_{i_1\cdots i_q}(y_1,\ldots,y_q,f(x_{1,k_1},\cdots,x_{q,k_q})^{**})\right)\\
					\leq &~ \max \bigg\{(\lvert \mu_{1,i_1}\rvert+\theta \mathcal{M}_{1}+\theta), \ldots, (\lvert \mu_{q,i_q}\rvert+\theta \mathcal{M}_{q}+\theta)\\
					&\hspace{5cm},  (\tfrac{\phi(\lvert x_1-y_1\rvert +\cdots +\lvert x_q-y_q\rvert+\lvert {f(x_{1,k_1},\cdots,x_{q,k_q})^*}-{f(x_{1,k_1},\cdots,x_{q,k_q})^{**}}\rvert)}{\lvert x_1-y_1\rvert +\cdots +\lvert x_q-y_q\rvert+\lvert {f(x_{1,k_1},\cdots,x_{q,k_q})^*}-{f(x_{1,k_1},\cdots,x_{q,k_q})^{**}}\rvert})\bigg\}\\
					&\hspace{4cm}\left(\lvert x_1-y_1\rvert +\cdots +\lvert x_q-y_q\rvert +\theta \lvert {f(x_{1,k_1},\cdots,x_{q,k_q})^*}-{f(x_{1,k_1},\cdots,x_{q,k_q})^{**}}\rvert\right).\\
					\leq&~ \beta\left(\mathcal{D}_{\theta}(x_1,\ldots,x_q,{f(x_{1,k_1},\cdots,x_{q,k_q})^*}),(y_1,\ldots,y_q,{f(x_{1,k_1},\cdots,x_{q,k_q})^{**}})\right)\\
					&\mathcal{D}_{\theta}\left((x_1,\ldots,x_q,{f(x_{1,k_1},\cdots,x_{q,k_q})^*}),(y_1,\ldots,y_q,{f(x_{1,k_1},\cdots,x_{q,k_q})^{**}})\right).
		\end{align*}}}
		
		\noindent
		Therefore, for each $(i_1,\ldots,i_q)\in \prod_{k=1}^{q}\Sigma_{N_k}$, $W_{i_1\cdots i_q}:X\rightarrow X$ is a Rakotch contraction map in $(X, \mathcal{D}_{\theta})$. Further, since $\mathcal{D}_{\theta}$ is equivalent to the Euclidean metric $\mathcal{D}_{0}$, hence $(X, \mathcal{D}_{\theta})$ is also a complete metric space.
		
		Thus, by Theorem \ref{rakotchthoerem} , there exist a unique function $f^{\phi} _{\Delta,b}\in \mathcal{C}^* (I^q)$ which satisfies the following self-referential equation 
		\begin{align*}
			f^{\phi} _{\Delta,b}(u_{1,i_1}(x_1),\cdots,u_{q,i_q}(x_q))=& f_{i_1,\cdots,i_q}(u_{1,i_1}(x_1),\cdots,u_{q,i_q}(x_q))\\
			&+r_{i_1,\cdots,i_q}((x_{1,k_1},\cdots,x_{q,k_q}),f^{\phi} _{\Delta,b}(x_{1,k_1},\cdots,x_{q,k_q}))\\
			&-r_{i_1,\cdots,i_q}((x_{1,k_1},\cdots,x_{q,k_q}),b(x_{1,k_1},\cdots,x_{q,k_q}))
		\end{align*}
		for all $(x_{1,k_1},\cdots,x_{q,k_q})\in I^q$ and  $(i_1,\cdots,i_q) \in \prod_{k=1}^{q}   \sum_{N_k}$. Using above self-referential equation satisfied by$f^{\phi} _{\Delta,b}$, one can prove that 
		\[{\mathcal{G}_{f^{\alpha} _{\Delta,b}}(I^q)}=\bigcup_{i_1=1}^{N_1}\cdots \bigcup_{i_q=1}^{N_q} W_{i_1\cdots i_q}(\mathcal{G}_{f^{\alpha} _{\Delta,b}}(I^q)).\]
		Thus, the proof is complete.
	\end{proof}
	
	\begin{remark}
		In the above construction, if we consider $r_{i_1,\cdots,i_q}((x_{1,k_1},\cdots,x_{q,k_q}),(y_{1,k_1},\cdots,y_{q,k_q}))=\alpha_{i_k}y_{(y_{1,k_1},\cdots,y_{q,k_q})} $ for all $i_k \in \sum_{M_k}$, where $\lvert \alpha_{i_k} \rvert<1$, then by our construction, we get the $\alpha$-fractal function for the multivariate case. If we take  $r_{i_1,\cdots,i_q}((x_{1,k_1},\cdots,x_{q,k_q}),(y_{1,k_1},\cdots,y_{q,k_q}))=\alpha_{i_k}(x_{1,k_1},\cdots,x_{q,k_q})\\
		(y_{1,k_1},\cdots,y_{q,k_q}) $ for all $i_k \in \sum_{M_k}$, where $\alpha_{i_k} : I^q \rightarrow \mathbb{R}$ is a continuous map such that $\sup_{(x_{1,k_1},\cdots,x_{q,k_q})\in I^q}\\
		\lvert \alpha_{i_k} \rvert<1$, then we get  $\alpha$-fractal functions  with variable scaling. On the other hand, if we consider $r_{i_1,\cdots,i_q}((x_{1,k_1},\cdots,x_{q,k_q}),(y_{1,k_1},\cdots,y_{q,k_q})) = \alpha_{i_k}(x_{1,k_1},\cdots,x_{q,k_q}) \frac{(y_{1,k_1},\cdots,y_{q,k_q})}{1+(y_{1,k_1},\cdots,y_{q,k_q})}$ or \\
		$\alpha_{i_k}(x_{1,k_1},\cdots,x_{q,k_q})\frac{1}{1+(y_{1,k_1},\cdots,y_{q,k_q})}$ or $\frac{(y_{1,k_1},\cdots,y_{q,k_q})}{1+(y_{1,k_1},\cdots,y_{q,k_q})}$ or $\frac{1}{1+(y_{1,k_1},\cdots,y_{q,k_q})}$, where $\alpha_{i_k} : I^q \rightarrow \mathbb{R}$ is a continuous map such that $\sup_{(x_{1,k_1},\cdots,x_{q,k_q})\in I^q}\lvert \alpha_{i_k} \rvert<1$, then we obtain some new type of multivariate fractal function.
	\end{remark}  
	
	\section{Invariant Measures for Rakotch Contractions}\label{sec5}	
	In this section, we study invariant probability measures associated with Rakotch contractions, as defined in Section~\ref{sec2}. We consider weighted IFSs acting on a complete metric space and analyze the corresponding Markov operators induced by probability. Using the contractive properties of Rakotch mappings, we prove the existence and uniqueness of an invariant Borel probability measure and demonstrate that its support is exactly the attractor of the IFS, i.e., the graph of the corresponding multivariate FIF.
	\begin{theorem}\label{measureRakotch}
		Let $I= \{X=\mathbb{I}^q \times \mathbb{R}, \Omega_{i_1,\cdots,i_q} ; (i_1,\cdots,i_q) \in \prod_{k=1}^{q} \sum_{N_k}\}$ be an IFS and each $\Omega_{i_1,\cdots,i_q}$ is a Rakotch contraction on $(X,D_\theta)$. Let $P_{i_1,\cdots,i_q}$ be a probability vector satisfying $\sum_{i_1=1}^{N_1} \cdots \sum_{i_q=1}^{N_q} P_{i_1,\cdots,i_q} =1$, then there exist a unique Borel probability measure $\mu^* \in \mathcal{P}(X)$ supported on the graph $G$ of IFS $I$ such that 
		\[\mu^* = \sum_{i_1=1}^{N_1} \cdots \sum_{i_q=1}^{N_q} P_{i_1,\cdots,i_q} \mu^* \circ \Omega^{-1}_{i_1,\cdots,i_q} \]
		Moreover, \[\operatorname{supp}(\mu^*) = \mathcal{G}\]
		where $\mathcal{G}$ is the attractor (graph) of the IFS $\mathbb{I}$.
	\end{theorem}
	
	\begin{proof}
		For each $(i_1,\cdots,i_q) \in \prod_{k=1}^{q} \sum_{N_k}$. The map $\Omega_{i_1,\cdots,i_q} : X \rightarrow X$ is a Rakotch contraction, then by \cite[Theorem $1$]{jachymski2007nonlinear}, there exists a strictly increasing, concave function, $\phi:\mathbb{R}_+\to\mathbb{R}_+$ such that
		\[ D_\theta\big(\Omega_{i_1,\cdots,i_q}(u), \Omega_{i_1,\cdots,i_q}(v)\big) \leq \phi(D_\theta(u,v)) ~~ \forall ~~ u,v \in X.\] 	
		Let $\mathcal{P}(X)$ denote the space of all Borel probability measures on $X$	with compact support. Now, define Monge Kantorovich metric on $\mathcal{P}(X)$
		\[
		d_{MK}(\mu,\nu) =\sup \left\{\int_X f \, d\mu - \int_X f \, d\nu:\;f : X \to \mathbb{R},\;\mathrm{Lip}_{D_\theta}(f) \le 1\right\},
		\]
		where $ \mu,\nu \in \mathcal{P}(X)$. Then, $(\mathcal{P}(X), d_{MK})$ is a complete metric space. By \cite[Theorem $8.10.45$]{bogachev2007measure}, $d_{MK}(\mu,\nu)$ can be represented in the following form
		\begin{equation}\label{markovother}
			d_{MK}(\mu,\nu)= \inf_{\lambda \in L(\mu,\nu)}
			\int_{X \times X} D_\theta (u,v)\, d\lambda(u,v).
		\end{equation}
		Moreover, there exists a $\lambda_0 \in  L(\mu,\nu)$ such that the infimum in \eqref{markovother} is attained. Therefore, 
		\[d_{MK}(\mu,\nu)= \int_{X \times X} D_\theta(u,v)\, d\lambda_0(u,v),\]
		where, $ L(\mu,\nu)$ denotes the space of all Borel probability measures $\lambda$ on $X \times X$ such that the projection of $\lambda$ on the first coordinate is $\mu$ and on the second coordinate is $\nu$.
		Now, we define a Markov Operator $ \mathbb{M}: \mathcal{P}(X) \rightarrow \mathcal{P}(X)$ as follows
		\[ \mathbb{M}(\mu) = \sum_{i_1=1}^{N_1} \cdots \sum_{i_q=1}^{N_q} P_{i_1,\cdots, i_q} \mu\circ \Omega^{-1}_{i_1,\cdots, i_q}.\]
		It is not difficult to show that $\mathbb{M}$ is well-defined. We aim to show that $\mathbb{M}$ is a Rakotch contraction map. We define $\lambda_1 = \sum_{i_1=1}^{N_1} \cdots \sum_{i_q=1}^{N_q} P_{i_1,\cdots, i_q} \lambda_0( \Omega^{-1}_{i_1,\cdots, i_q} \times \Omega^{-1}_{i_1,\cdots, i_q})$. 
		Let $A$ be any Borel subset of $X$. Then, we have 
		\begin{align*}
			\lambda_1(X,A) =& \sum_{i_1=1}^{N_1} \cdots \sum_{i_q=1}^{N_q} P_{i_1,\cdots, i_q} \lambda_0( \Omega^{-1}_{i_1,\cdots, i_q} \times \Omega^{-1}_{i_1,\cdots, i_q}) (X,A)\\
			=& \sum_{i_1=1}^{N_1} \cdots \sum_{i_q=1}^{N_q} P_{i_1,\cdots, i_q} \lambda_0( \Omega^{-1}_{i_1,\cdots, i_q} (X) \times \Omega^{-1}_{i_1,\cdots, i_q}(A))\\
			=& \sum_{i_1=1}^{N_1} \cdots \sum_{i_q=1}^{N_q} P_{i_1,\cdots, i_q} \lambda_0( X , \Omega^{-1}_{i_1,\cdots, i_q}(A))
		\end{align*}
		Since $\lambda_0 \in L(\mu,\nu)$, we have 
		\[	\lambda_1(X,A) = \sum_{i_1=1}^{N_1} \cdots \sum_{i_q=1}^{N_q} P_{i_1,\cdots, i_q} \nu ( \Omega^{-1}_{i_1,\cdots, i_q} (A)) = (\mathbb{M}(\nu))(A).
		\]
		Similarly, $\lambda_1(A,X)=(\mathbb{M}(\mu))(A) $. Therefore, $\lambda_1 \in L(\mathbb{M}(\mu),\mathbb{M}(\nu)).$\\
		We have
		\begin{align*}
			d_{MK}(\mathbb{M}(\mu),\mathbb{M}(\nu)) \leq& \int_{X \times X} D_\theta (u,v) d\lambda_1 (u,v)\\
			=& \sum_{i_1=1}^{N_1} \cdots \sum_{i_q=1}^{N_q} P_{i_1,\cdots, i_q} \int_{X \times X} D_\theta( \Omega_{i_1,\cdots, i_q} (u) , \Omega_{i_1,\cdots, i_q}(v))d\lambda_0 (u,v).
		\end{align*}
		Since, each $\Omega_{i_1,\cdots, i_q}$ is a Rakotch contraction.	
		Therefore, 
		\[ d_{MK}(\mathbb{M}(\mu),\mathbb{M}(\nu)) \leq  \sum_{i_1=1}^{N_1} \cdots \sum_{i_q=1}^{N_q} P_{i_1,\cdots, i_q} \int_{X \times X}\phi(D_\theta(u,v))d\lambda_0 (u,v).
		\] 
		Since $\phi$ is concave and non-decreasing, then by using definition \ref{jensendefinition}, we have
		\begin{align*}
			\int_{X \times X}\phi(D_\theta(u,v))d\lambda_0 (\mu,\nu) \leq&~ \phi\left(\int_{X \times X} D_\theta(u,v)d\lambda_0 (u,v)\right)\\
			\leq&~ \phi(d_{MK}(\mu,\nu))
		\end{align*}
		Therefore, 
		\[d_{MK}(\mathbb{M}(\mu),\mathbb{M}(\nu)) \leq \sum_{i_1=1}^{N_1} \cdots \sum_{i_q=1}^{N_q} P_{i_1,\cdots, i_q} \phi(d_{MK}(\mu,\nu)) = \boldsymbol{\Phi}(d_{MK}(\mu,\nu)), \]
		where $\boldsymbol{\Phi} : \mathbb{R}^+ \to \mathbb{R}^+$ is defined by
		\[\boldsymbol{\Phi(t)} = \sum_{i_1=1}^{N_1} \cdots \sum_{i_q=1}^{N_q}P_{i_1,\ldots,i_q}\,\phi(t).\]
		By the definition of $\phi$ and the normalization condition
		\[\sum_{i_1=1}^{N_1} \cdots \sum_{i_q=1}^{N_q}P_{i_1,\ldots,i_q} = 1,\] we get that $\boldsymbol{\Phi}$ is strictly increasing and concave. Hence, the Markov operator $\mathbb{M}$ is a Rakotch contraction on the metric space $(\mathcal{P}(X), d_{MK})$. By the Rakotch fixed point theorem, $\mathbb{M}$ has unique fixed point $\mu^* \in \mathcal{P}(X)$, i.e., $\mathbb{M}(\mu^*)=\mu^*$. Thus we have 
		\[\mu^* = \sum_{i_1=1}^{N_1} \cdots \sum_{i_q=1}^{N_q} P_{i_1,\cdots,i_q} \mu^* \circ \Omega^{-1}_{i_1,\cdots,i_q}.\]
		Now, we show that $\operatorname{supp}(\mu^*)= G $. To show this, it is enough to show that $\operatorname{supp}(\mu^*)$ is the attractor of the IFS $\mathbb{I}$. Firstly, we show that $\operatorname{supp}(\mu^*) \subseteq  \bigcup_{i_1=1}^{N_1}\cdots \bigcup_{i_q=1}^{N_q} \Omega_{i_1\cdots i_q}\operatorname{supp}(\mu^*).$
		\begin{align*}
			\mu^*\left( \bigcup_{i_1=1}^{N_1}\cdots \bigcup_{i_q=1}^{N_q} \Omega_{i_1\cdots i_q}\operatorname{supp}(\mu^*)\right) =&  \sum_{i_1=1}^{N_1} \cdots \sum_{i_q=1}^{N_q} P_{i_1,\cdots,i_q} \mu^*\left(\Omega^{-1}_{i_1\cdots i_q}(\bigcup_{i_1=1}^{N_1}\cdots \bigcup_{i_q=1}^{N_q} \Omega_{i_1\cdots i_q}(\operatorname{supp}(\mu^*)))\right)\\
			\geq& \sum_{i_1=1}^{N_1} \cdots \sum_{i_q=1}^{N_q} P_{i_1,\cdots,i_q} \mu^*\left(\Omega^{-1}_{i_1\cdots i_q}( \Omega_{i_1\cdots i_q}(\operatorname{supp}(\mu^*)))\right)
		\end{align*}
		Since $\sum_{i_1=1}^{N_1} \cdots \sum_{i_q=1}^{N_q} P_{i_1,\cdots, i_q}=1$ and $\mu^*(\operatorname{supp}(\mu^*)) = 1$, we have
		\[	\mu^*\left( \bigcup_{i_1=1}^{N_1}\cdots \bigcup_{i_q=1}^{N_q} \Omega_{i_1\cdots i_q}\operatorname{supp}(\mu^*)\right) = 1.\]
		Therefore, $\operatorname{supp}(\mu^*) \subseteq  \bigcup_{i_1=1}^{N_1}\cdots \bigcup_{i_q=1}^{N_q} \Omega_{i_1\cdots i_q}\operatorname{supp}(\mu^*).$\\
		Now we prove that $\bigcup_{i_1=1}^{N_1}\cdots \bigcup_{i_q=1}^{N_q} \Omega_{i_1\cdots i_q}\operatorname{supp}(\mu^*) \subseteq \operatorname{supp}(\mu^*)$. we have 
		\begin{align*}
			1= \sum_{i_1=1}^{N_1} \cdots \sum_{i_q=1}^{N_q} P_{i_1,\cdots, i_q} \geq \sum_{i_1=1}^{N_1} \cdots \sum_{i_q=1}^{N_q} P_{i_1,\cdots,i_q} \mu^*\left(\Omega^{-1}_{i_1\cdots i_q}(\operatorname{supp}(\mu^*))\right)= \mu^*(\operatorname{supp}(\mu^*))=1
		\end{align*}
		Since $\sum_{i_1=1}^{N_1} \cdots \sum_{i_q=1}^{N_q} P_{i_1,\cdots, i_q}=1$ and $\sum_{i_1=1}^{N_1} \cdots \sum_{i_q=1}^{N_q} P_{i_1,\cdots,i_q} \mu^*\left(\Omega^{-1}_{i_1\cdots i_q}(\operatorname{supp}(\mu^*))\right)=1,$ we have $\mu^*\left(\Omega^{-1}_{i_1\cdots i_q}(\operatorname{supp}(\mu^*))\right)=1$ for each $(i_1,\cdots, i_q)$. \\
		Therefore, for every $(i_1,\cdots, i_q)$, $\operatorname{supp}(\mu^*) \subseteq \Omega^{-1}_{i_1\cdots i_q}(\operatorname{supp}(\mu^*)).$ This implies that $\Omega_{i_1\cdots i_q}\operatorname{supp}(\mu^*) \subseteq \operatorname{supp}(\mu^*).$ This yields that 
		$\bigcup_{i_1=1}^{N_1}\cdots \bigcup_{i_q=1}^{N_q} \Omega_{i_1\cdots i_q}\operatorname{supp}(\mu^*) \subseteq \operatorname{supp}(\mu^*).$ Thus, we get 
		\[\operatorname{supp}(\mu^*) = \bigcup_{i_1=1}^{N_1}\cdots \bigcup_{i_q=1}^{N_q} \Omega_{i_1\cdots i_q}\operatorname{supp}(\mu^*).  \]
		Since the IFS $\mathbb{I}$ possesses a unique attractor $\mathcal{G}$, this implies that $\operatorname{supp}(\mu^*)$ is the attractor of the IFS $\mathbb{I}$. we conclude that $\operatorname{supp}(\mu^*) = \mathcal{G}$. This completes the proof.

	\end{proof}	
	\begin{remark}
		The proof for the above Theorem \eqref{measureRakotch} is based on the $\phi$ function, which is strictly increasing and concave. Consequently, it remains an open problem whether the existence and uniqueness of an invariant probability measure can be established for IFSs generated by Matkowski contractions.
	\end{remark}
	
	\section{Conclusion and Future directions}\label{sec3}
	In this paper, we have developed the concept of nonlinear multivariate FIFs. To construct the nonlinear multivariate FIF, we have used the Matkowski and the Rakotch contractions, a more general contraction, instead of the Banach contraction, which is a particular case of the Matkowski and the Rakotch contractions.
	We have proposed a method to construct nonlinear multivariate IFSs using Matkowski and Rakotch contractions. Further, we have proved that the attractors of nonlinear IFS constructed using Matkowski (Theorem \ref{Matkowski}) and Rakotch (Theorem \ref{rakotchthoerem}) contractions are graphs of some continuous functions (Lemma \ref{operator}) which interpolate the given data. Since Banach contraction is a special case of Matkowski and Rakotch contraction, it is easy to observe that our method is a more generalized version of the so far published methods for constructing the multivariate FIFs. Furthermore, we extended the proposed framework to construct $\alpha$-fractal functions associated with multivariate continuous functions by employing both Matkowski (Theorem \ref{alphaMatkowski}) and Rakotch (Theorem \ref{alphaRakotch}) contractions. We proved the existence and uniqueness of such $\alpha$-fractal functions under these generalized contractive conditions, thereby broadening the scope of $\alpha$-fractal function theory beyond the classical Banach setting.
	
	In addition, we investigate invariant probability measures associated with nonlinear multivariate iterated function systems generated by Rakotch contractions (Theorem \ref{measureRakotch}) and establish the existence and uniqueness of invariant measures supported on the graph of the multivariate fractal interpolation function.
	
	Further, we have worked on constructing multivariate FIFs using the Matkowski and the Rakotch contractions.  Numerous results have been given regarding the fractal dimension of Barnsley's type FIFs \cite{jha2021dimensional,pandey2022fractal}. In the future, one may take up to generalize the concept of fractal dimension for this new FIF and work on the fractal dimensional aspects of such fractal functions.
	
	Further, fractional calculus has been explored for Barnsley's type FIF. See, for instance, \cite{liang2010box,chandra2021calculus}. One may also attempt to study the fractional calculus of this new FIF.

	\section*{Declaration}
	
	\textbf{Conflicts of interest.} We do not have any conflict of interest.\\
	\\
	\noindent
	\textbf{Data availability:} No data were used to support this study.\\
	\\
	\noindent
	\textbf{Code availability:} Not applicable\\
	\\
	\noindent
	\textbf{Authors' contributions:} Each author contributed equally to this manuscript.
	
	\section*{Acknowledgements}
	This work is supported by a Postdoctoral fellowship by Northwest University, Xi'an, Shaanxi Province, China, to the first author. Further, supported by UGC $(231620024614)$ to the second author.

	\bibliographystyle{spmpsci} 
	\bibliography{Reference}
	
\end{document}